\documentclass{article}
\usepackage{fullpage, amsfonts, amsthm, amsmath, setspace, graphicx, amssymb, float}
\usepackage{colonequals}

\newtheorem{theorem}{Theorem}[section]
\newtheorem{lemma}[theorem]{Lemma}
\newtheorem{fact}[theorem]{Fact}
\newtheorem{proposition}[theorem]{Proposition}
\newtheorem{corollary}[theorem]{Corollary}
\newtheorem{case}{Case}[theorem]

\newtheorem{claim}{Claim}[theorem]
\numberwithin{subcase}{case}
\numberwithin{claim}{theorem}

\newenvironment{definition}[1][Definition:]{\begin{trivlist}
\item[\hskip \labelsep {\bfseries #1}]}{\end{trivlist}}
\newenvironment{example}[1][Example:]{\begin{trivlist}
\item[\hskip \labelsep {\bfseries #1}]}{\end{trivlist}}

\newenvironment{note}[1][Note:]{\begin{trivlist}
\item[\hskip \labelsep {\bfseries #1}]}{\end{trivlist}}
\title{\bf Notes on the Level Curves of a Meromorphic Function}
\author{Trevor Richards\thanks{Email: trevor1r@ufl.edu}\\\vspace{6pt} {\em Department of Mathematics, University of Florida, Gainesville, United States}}

\begin{document}

\maketitle

\begin{abstract}
The subject of this paper is the bounded level curves of a meromorphic function $f$ with domain $G$ such that each component of $\partial{G}$ consists of a level curve of $f$.  (A primary example of such a function being a ratio of finite Blaschke products of different degrees, with domain $\mathbb{D}$.)  We will first prove several facts about a single bounded level curve of a $f$ in isolation from the other level curves of $f$.  We will then study how the level curves of $f$ lie with respect to each other.  It is natural to expect that the sets $\{z:|f(z)|=\epsilon\}$ vary continuously as $\epsilon$ varies.  We will make this notion explicit, and use this continuity to prove several results about the bounded level curves of $f$.  It is well known that if $z_0$ is a zero or a pole of $f$, then $f$ is conformally equivalent to the function $z\mapsto{z^k}$ (for some $k\in\mathbb{Z}$) in a neighborhood of $z_0$.  We generalize this fact by finding a natural decomposition of $G$ into finitely many sub-regions (also bounded by level curves of $f$), on each of which $f$ is conformally equivalent to $z\mapsto{z^k}$ (for some $k\in\mathbb{Z}$).  Also included is a new proof, using level curves, of the Gauss--Lucas theorem that the critical points of a polynomial are contained in the convex hull of the polynomial's zeros.\medskip

\end{abstract}

\section{Setting, Background, and Major Results}\label{Setting, Several Definitions, and Major Results}

The study of the level curves of a meromorphic function $f$ (sets of the form $\{z:|f(z)|=\epsilon\}$) or tracts of the function (sets of the form $\{z:|f(z)|<\epsilon\}$) has seemed to cluster around two major classes of questions.  We will briefly mention a couple of results from each class to give a taste of the existing research on this subject.  The first class of questions has to do with properties of a level curve (or tract) of a single function such as arc length of the level curve, area of the tract, and convexity or star-shapeness of the tract.  In 1981, W. K. Hayman and J. M. Wu~\cite{HW} showed that if $f:\mathbb{D}\to\mathbb{C}$ is an injective analytic function, then any level set of $f$ has length less than $2\cdot{10}^{35}$.  On the other hand, in 1980, P. W. Jones~\cite{J} constructed an analytic bounded function $f:\mathbb{D}\to\mathbb{C}$ each of whose level sets is either empty or of infinite length.  Questions about convexity and star-shapeness have mostly had to do with polynomial tracts, and several counter-examples have been given by G. Pirinian and others~\cite{EHP,G,P}.

The other class of questions has to do with meromorphic functions which share a level curve or a tract.  A number of people have given results showing the relationship between two meromorphic functions which share a level curve, starting with G. Valiron \cite{V} in 1937, whose work implied the following result.

\begin{theorem}
Let $f_1$ and $f_2$ be entire functions such that $|f|\equiv1\equiv|f_2|$ on a simple closed curve $\Gamma$.  Then there is an entire function $h$ with $|h|\equiv1$ on $\Gamma$, and $f_1\equiv{b_1}\circ{h}$ and $f_2\equiv{b_2}\circ{h}$ for some finite Blaschke products $b_1$ and $b_2$.
\end{theorem}

A nice summation of this result and following results of the same flavor may be found in a 1983 paper by K. Stephenson and C. Sundberg~\cite{StSu} in which they give the following stronger result for inner functions.

\begin{theorem}
Let $\phi_1$ and $\phi_2$ be inner functions such that the sets $\{z:|\phi_1(z)|=c\}$ and $\{z:|\phi_2(z)|=c\}$ have a sub-arc in common for some $c\in(0,1)$.  Then $\phi_1=\lambda\phi_2$ for some unimodular constant $\lambda$.
\end{theorem}

Since then K. Stephenson~\cite{St} has given a more general result which implies much of the earlier work.

In future work we intend to prove a result somewhat in the vein of the results from this second class of questions.  Namely we will show that if the configuration of the critical level curves of two finite Blaschke products $f$ and $g$ (that is, the level curves of $f$ and $g$ which contain critical points of the respective functions) are the same (modulo homeomorphism), then $f$ and $g$ are conformally equivalent (that is, there is some conformal $\phi:\mathbb{D}\to\mathbb{D}$ such that $g\equiv{f}\circ\phi$ on $\mathbb{D}$).  Some of the motivation behind this paper is to develop the preliminaries necessary for this future work, however the results found here are of independent interest as well.  The proofs found in this paper are elementary, using well known results about meromorphic functions such as the Maximum Modulus Theorem, the Open Mapping Theorem, and the Argument Principle.  These and the other facts needed may be found in many texts on complex analysis including~\cite{Alf} and~\cite{C1}.

In order to introduce the work in this paper, we make the following standard definition.

\begin{definition}
For $D\subset\mathbb{C}$ an open set, and $f:D\to\bar{\mathbb{C}}$ a meromorphic function, and $\epsilon\in[0,\infty]$, the set $E_{f,\epsilon}\colonequals\{z\in{D}:|f(z)|=\epsilon\}$ is called a level set of $f$, and a component of $E_{f,\epsilon}$ is called a level curve of $f$.
\end{definition}

We now establish the setting used in Section~\ref{Level Curves as Planar Graphs} in order to describe the results found therein.  Let $G\subset\mathbb{C}$ be an open set.  Let $f$ be a function meromorphic on an open set containing $cl(G)$.  Let $\epsilon\in(0,\infty)$ be given, and let $\Lambda\subset{G}$ be a fixed bounded component of $E_{f,\epsilon}$ which is contained in $G$.  In Section~\ref{Level Curves as Planar Graphs}, we examine the single level curve $\Lambda$.  We observe that $\Lambda$ forms a planar graph in $\mathbb{C}$, whose vertices are critical points of $f$, and we determine exactly which planar graphs we might see as bounded level curves of $f$.  As an application of this, we give a new proof using level curves of the Gauss-Lucas theorem that the critical points of a polynomial are contained in the convex hull of the polynomial's zeros.  We then give a formula for counting the faces of $\Lambda$ in terms of the multiplicities of the critical points of $f$ in $\Lambda$.

In Section~\ref{Assorted Properties of Level Curves} we impose the following further restrictions of $G$ and $f$.

\begin{itemize}
	\item
	$|f|\equiv1$ on any unbounded component of $\partial{G}$.  For each bounded component $K$ of $\partial{G}$, there is some $\rho\in(0,\infty)$ such that $|f|\equiv\rho$ on $K$.  (That is, each component of $\partial{G}$ is contained in some level curve of $f$ in $G'$.)

	\item
	If $G$ is unbounded, then $\displaystyle\lim_{z\in{G},z\to\infty}|f(z)|=1$.
	
	\item
	$G$ does not contain any unbounded level curve of $f$.

	\item
	$G$ is connected and $G^c$ has finitely many bounded components.

	\item
	No level curve of $f$ in $G$ intersects $\partial{G}$.
\end{itemize}

It seems reasonable to expect that as the "height" $\epsilon$ of $f$ varies continuously, the corresponding level curves of $f$ will in some sense vary continuously as well.  Proposition~\ref{Level Curves Vary Continuously} makes this notion of continuous variance explicit.  Theorem~\ref{Two Level Curve Implies Crit. Level Curve.} shows that if there are any two level curves $L_1$ and $L_2$ of $f$ in $G$, each of which is exterior to the other, then there is a critical point of $f$ in $G$ which is exterior to both $L_1$ and $L_2$.  This implies that if we remove from $G$ each level curve of $f$ which contains a critical point, and each zero and pole of $f$ in $G$, then each component of the remaining set is conformally equivalent to an annulus.  In order to explain the final major result of Section~\ref{Assorted Properties of Level Curves}, we first recall a property of meromorphic functions.

It is well know that if $g$ is a function meromorphic on some region $D\subset\mathbb{C}$, and $z\in{D}$ is a zero or pole of $g$, then there is some region $D_z\subset{D}$ containing $z$, and $\phi$ a conformal mapping which maps $D_z$ to a disk centered at $0$ such that $f\equiv\phi^n$ on $D_z$ for some $n\in\mathbb{Z}\setminus\{0\}$.  Our final major result in Section~\ref{Assorted Properties of Level Curves} extends this fact by in a sense decomposing $G$ into finitely many regions on each of which $f$ is conformally equivalent to a pure power of $z$.  That is, if $A$ is one of the regions remaining after removing the level curves of $f$ in $G$ which contain critical points of $f$, along with the zeros and poles of $f$ in $G$, as noted above $A$ is conformally equivalent to an annulus.  In Theorem~\ref{f Conformally Equiv. to a Power of z} we find explicitly which annulus $A$ is conformally equivalent to, and construct a conformal mapping $\phi$ from $A$ to this annulus, such that for some $n\in\mathbb{Z}\setminus\{0\}$, $f\equiv\phi^n$ on $A$.

\section{Level Curves as Planar Graphs}\label{Level Curves as Planar Graphs}

\begin{definition}
By a planar graph we mean a finite graph embedded in a plane or sphere, whose edges do not cross.  For the purposes of this paper we will include simple closed paths as planar graphs.
\end{definition}

In the existing literature on the subject of level curves of a meromorphic function, it is generally assumed without proof that $\Lambda$ is a planar graph whose vertices are the critical points of $f$ in $\Lambda$, and whose edges are paths in $\Lambda$ with endpoints at the critical points of $f$.  However because part of the subject of this paper is precisely the geometry of $\Lambda$, we make that fact explicit in Proposition~\ref{Lambda is a planar graph.}, and give an indication as to the nature of the proof.  As much of the proof amounts to an exercise in analytic continuation and compactness, we leave out most of the details.

\begin{proposition}\label{Lambda is a planar graph.}
$\Lambda$ is a planar graph.
\end{proposition}

\begin{proof}
Since $\Lambda$ is bounded and $f$ is non-constant, there are only finitely many critical points of $f$ in $\Lambda$.  It will turn out that the critical points of $f$ are exactly the intersections of $\Lambda$.  Choose some $z_0\in\Lambda$ which is not a critical point of $f$.  We wish to show that either $z_0$ lies in a path in $\Lambda$ whose endpoints are critical points of $f$, or that $\Lambda$ contains no critical points of $f$, in which case $\Lambda$ consists of a simple closed path.

Since $f'(z_0)\neq0$, there is some open simply connected set $D$ which contains $z_0$ such that an inverse $g$ of $f$ may be defined on $f(D)$.  $f(z_0)$ lies in a circle of radius $\epsilon$.  Our strategy is to continue $g$ analytically around that circle as far as possible, and consider the path that this traces out back in the domain of $f$.

\begin{case}\label{g can be continued indefinitely along circle}
$g$ may be continued indefinitely around the $\epsilon$-circle in both directions.
\end{case}

By the compactness of $\Lambda$, there are only finitely many points in $\Lambda$ at which $f$ takes the value $f(z_0)$.  Therefore if $g$ is continued analytically along the $\epsilon$-circle in one direction enough times, the path traced out in the domain intersects itself.  Let $\sigma$ denote this path.  Because $f$ is invertible at each point in $\sigma$ we may conclude that $\sigma$ is simple (as f would not be injective at a crossing point of such a path) and does not contain any critical point of $f$.  Furthermore, we can use the invertibility of $f$ at each point in $\sigma$, and the connectedness of $\Lambda$, to show that the path $\sigma$ is all of $\Lambda$.  Thus $\Lambda$ is a simple closed path which does not contain any critical points of $f$.

\begin{case}
$g$ may not be continued indefinitely around the $\epsilon$-circle in at least one direction.
\end{case}

Again let $\sigma$ denote the path obtained by continuing $g$ around the $\epsilon$-circle as far as may be in the direction in which it cannot be continued indefinitely.  Clearly $\sigma$ is contained in $\Lambda$ and does not contain any critical points of $f$ (because $f$ is invertible at each point in $\sigma$).  One can use the compactness of $\Lambda$ to show that $\sigma$ approaches a unique point in $\Lambda$, and one can show that this point is a critical point of $f$ (otherwise one could continue $g$ further along the $\epsilon$-circle).  $g$ may not be continued indefinitely around the $\epsilon$ circle in the opposite direction because this would imply that $\Lambda$ does not contain a critical point of $f$ (as seen in Case~\ref{g can be continued indefinitely along circle} above).  Therefore if we allow $\sigma$ now to denote the path obtained by continuing $g$ along the $\epsilon$-circle in the opposite direction, we obtain another path from $z_0$ to a critical point of $f$.  We conclude that $z_0$ lies in a path in $\Lambda$ whose end points are critical points of $f$.  Furthermore this path is simple because $f$ is invertible at each point in this path other than the end points.

Finally we may again use the compactness of $\Lambda$ to show that $\Lambda$ contains only finitely many edges and vertices, which gives us the desired result.
\end{proof}

By inspecting the above proof we may make the following observation.

\begin{corollary}\label{Vertices are Crit. Points of f.}
The critical points of $f$ contained in $\Lambda$ are the vertices of $\Lambda$.  If $\Lambda$ does not contain any critical point of $f$, then $\Lambda$ is a simple closed path.
\end{corollary}

Due to the nice behavior of rational functions near $\infty$, one can easily extend the result of the above proposition to all level curves of a rational function.  Since the proof merely requires one to pre-compose the rational function with a M\"obius function, we omit it here.

\begin{corollary}\label{Each Level Curve of Rat. Function is Planar Graph.}
If $g:\bar{\mathbb{C}}\to\bar{\mathbb{C}}$ is a rational function, every level curve of $g$ is a planar graph.
\end{corollary}

We will now uncover some restrictions on which planar graphs we may see as level curves of our function $f$.  First in Proposition~\ref{Each Edge of Lambda Adjacent to Bounded Face.} we see that $\Lambda$ must have bounded faces, and that each edge of $\Lambda$ is adjacent to at least one of its bounded faces.  In Proposition~\ref{Evenly Many Edges Incident to Each Vertex in Lambda.}, we see that each vertex of $\Lambda$ must have evenly many edges of $\Lambda$ incident to it.  First some topology.

\begin{lemma}\label{Closed Sets Topological Lemma}
For any disjoint closed sets $X,Y\subset\bar{\mathbb{C}}$, and $x,y\in\bar{\mathbb{C}}\setminus(X\cup{Y})$, if $x$ and $y$ are in the same component of $X^c$ and the same component of $Y^c$, then $x$ and $y$ are in the same component of $(X\cup{Y})^c$.
\end{lemma}

\begin{proof}
Given $X,Y,x,y$ as in the statement of the lemma, suppose by way of contradiction that $x$ and $y$ are in different components of $(X\cup{Y})^c$.  Assume without loss of generality that $y=\infty$.  Let $A_1$ denote the component of $(X\cup{Y})^c$ which contains $x$, and let $B_1$ denote the component of $(X\cup{Y})^c$ which contains $y$.  Let $Z$ denote the union of all bounded components of ${A_1}^c$.  Define $A_2\colonequals{A_1\cup{Z}}$.  Since $B_1$ is open and contains $\infty$, and $A_2$ does not intersect $B_1$, we may conclude that $A_2$ is bounded.  Therefore $A_2$ has only a single unbounded component, so $A_2$ is simply connected.  And the boundary of a simply connected set is connected, so $\partial{A_2}\subset{X\cup{Y}}$ is connected.  Because $X$ and $Y$ are disjoint and compact, this implies that $\partial{A_2}$ is either contained in $X$ or contained in $Y$.  However this is a contradiction because $x$ and $y$ are in the same component of $X^c$ and in the same component of $Y^c$.
\end{proof}

It may easily be seen that this lemma implies that if $A$ is an open connected set, and $X\subset{A}$ is compact, such that $X^c$ has a single component, then $A\setminus{X}$ is connected.  This fact will be used several times in this paper, and we will cite the above lemma when it is needed.  We now proceed to the proposition.

\begin{proposition}\label{Each Edge of Lambda Adjacent to Bounded Face.}
$\Lambda$ has one or more bounded faces and a single unbounded face, and each edge of $\Lambda$ is adjacent to two distinct faces of $\Lambda$.
\end{proposition}

\begin{proof}
Since $\Lambda$ is bounded, there is a single unbounded face of $\Lambda$, so to show that $\Lambda$ has bounded faces, it suffices to show the second result of the proposition, that each edge of $\Lambda$ is adjacent to two distinct faces of $\Lambda$.  Define $E\colonequals\{z\in{E_{f,\epsilon}\cap{G}}:z\notin\Lambda\}$.  Then $cl(E)$ is a closed set contained in $cl(G)$, and $cl(E)$ does not intersect $\Lambda$.

Choose some $z_0\in\Lambda$ which is not a zero of $f'$.  That is, $z_0$ is a point in one of the edges of $\Lambda$.  Choose some $\delta>0$ small enough so that $B_{\delta}(z_0)\subset{G\setminus{E}}$.  (So for all $z\in{B_{\delta}(z_0)}$, if $|f(z)|=\epsilon$, then $z\in\Lambda$.)  By the Open Mapping Theorem, there are points $x,y\in{B_{\delta}(z_0)}$ such that $|f(x)|<\epsilon$ and $|f(y)|>\epsilon$.  Suppose by way of contradiction that $x$ and $y$ are in the same face of $\Lambda$.  Lemma~\ref{Closed Sets Topological Lemma} (with $\Lambda=X$ and $E\cup{G^c}=Y$) gives that $x$ and $y$ are in the same component of $(\Lambda\cup{E}\cup{G^c})^c$ which is equal to $G\setminus{E_{f,\epsilon}}$.  Since open connected sets in $\mathbb{C}$ are path connected, there is a path $\gamma:[0,1]\to{G\setminus{E_{f,\epsilon}}}$ from $x$ to $y$.  But $f$ is continuous, and $|f(x)|<\epsilon<|f(y)|$, so there is some $s\in(0,1)$ such that $|f(\gamma(s))|=\epsilon$, and thus $\gamma(s)\in{E_{f,\epsilon}}$, and we have our contradiction.  So we conclude that $x$ and $y$ are in different faces of $\Lambda$, and thus $\Lambda$ has bounded faces.

Furthermore, since $\delta$ was arbitrarily small, and the points $x$ and $y$ we found in $B_{\delta}(z_0)$ were shown to be in different components of $\Lambda^c$, we have that $z_0$ is in the boundary of more than one component of $\Lambda^c$.  Since the edge of $\Lambda$ containing $z_0$ is locally conformally equivalent to the $\epsilon$-circle, we have that the edge of $\Lambda$ containing $z_0$ is in the boundary of at most two faces of $\Lambda$, and we are done.
\end{proof}

One more restriction may be made on which planar graphs we may see as level curves of $f$.  As the method of proof is largely the same as that used in the proof of Proposition~\ref{Lambda is a planar graph.}, we will leave out most of the details.  First a definition.

\begin{definition}
If $z\in{G}$ is a zero of $f'$, then let $mult(z)$ denote the multiplicity of $z$ as a zero of $f'$.
\end{definition}

\begin{proposition}\label{Evenly Many Edges Incident to Each Vertex in Lambda.}
If $z$ is a vertex of $\Lambda$, then the number of edges of $\Lambda$ which are incident to $z$ is exactly $2(mult(z)+1)$.
\end{proposition}

\begin{proof}
As already noted, if $z\in\Lambda$ is not a zero of $f'$, then $z$ is not a vertex of $\Lambda$.  Suppose $z\in\Lambda$ is a zero of $f'$.  Then $f$ is precisely $(mult(z)+1)\text{-to-}1$ in a neighborhood of $z$.  One can show (using the same tricks found in the proof of Proposition~\ref{Lambda is a planar graph.} of analytic continuation of branches of the inverse of $f$) that there are precisely $2(mult(z)+1)$ edges of $\Lambda$ incident to $z$.
\end{proof}

In future work we will show that for any planar graph which satisfies the restrictions described in the above propositions (namely a planar graph $\Lambda$ such that each edge is in the boundary of two distinct faces of $\Lambda^c$ and each vertex of $\Lambda$ has an evenly many, and four or more, edges incident to it), there is some some region $G$ and function $f$ as above such that some level curve of $f$ in $G$ is homeomorphic to the graph $\Lambda$.  In Figure~\ref{fig:AdmissibleGraphs} we have several examples of planar graphs that satisfy the above restrictions, and in Figure~\ref{fig:NonAdmissibleGraphs} we have several examples of planar graphs which do not satisfy these restrictions.

\begin{figure}[H]
\begin{minipage}[b]{0.45\linewidth}
\centering
	\includegraphics[width=\textwidth]{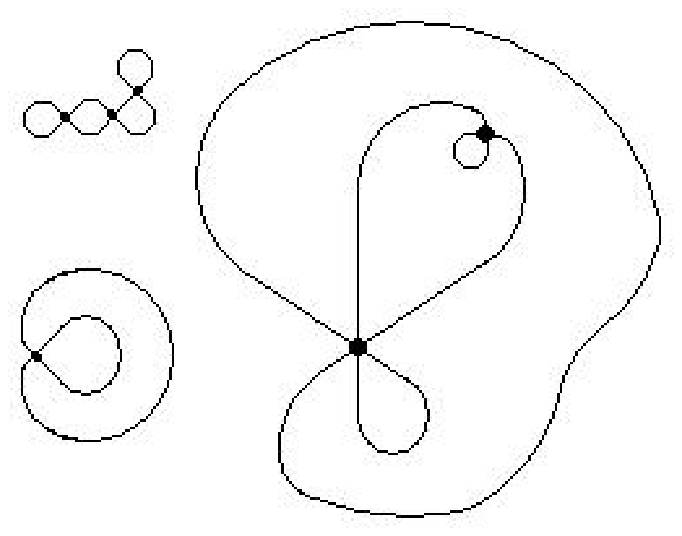}
	\caption{Admissible Graphs}
	\label{fig:AdmissibleGraphs}
\end{minipage}
\hspace{0.5cm}
\begin{minipage}[b]{0.45\linewidth}
\centering
	\includegraphics[width=\textwidth]{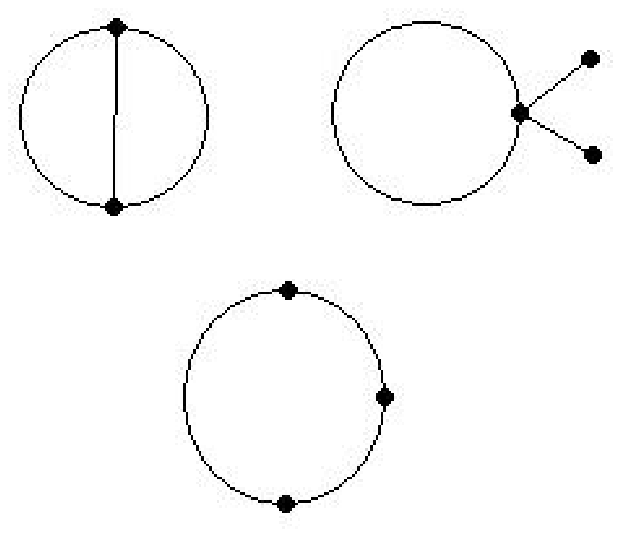}
	\caption{Non Admissible Graphs}
	\label{fig:NonAdmissibleGraphs}
\end{minipage}
\end{figure}

\begin{definition}
For $K\subset{G}$, define $B_K\colonequals\{z\in{K}:f'(z)=0,f(z)\neq0\}$.
\end{definition}

We are now able to count precisely the number of faces of $\Lambda$ in terms of the critical points of $f$ contained in $\Lambda$ as follows.

\begin{proposition}\label{Count of Faces of Lambda.}
$\Lambda$ has exactly $\left(\displaystyle\sum_{a\in{B_{\Lambda}}}mult(a)\right)+1$ bounded faces and one unbounded face.
\end{proposition}

\begin{proof}
Note first that if $\Lambda$ does not contain any zeros of $f'$ (so $B_{\Lambda}$ is empty), then from Corollary~\ref{Vertices are Crit. Points of f.}, $\Lambda$ is a single simple closed path, and thus $\Lambda$ has two faces, one bounded and one unbounded, and this agrees with the formula in the statement above.

Now suppose that $B_{\Lambda}$ is non-empty.  Let $\mathcal{V}$ denote the number of points in $B_{\Lambda}$ (therefore $\mathcal{V}$ is the number of vertices of $\Lambda$).  For $a\in{B_{\Lambda}}$, $f'$ has a zero with multiplicity $mult(a)$ at $a$, so by Proposition~\ref{Evenly Many Edges Incident to Each Vertex in Lambda.}, exactly $2(mult(a)+1)$ edges of $\Lambda$ meet at $a$.  Each line in $\Lambda$ has two endpoints, so this implies that there are $\dfrac{1}{2}\displaystyle\sum_{a\in{B_{\Lambda}}}2(mult(a)+1)=\left(\displaystyle\sum_{a\in B_{\Lambda}}mult(a)\right)+\mathcal{V}$ edges in $\Lambda$.

Euler's characteristic formula (see for example~\cite{Ale}) for a connected planar graph states that for any planar graph,

\[
\text{number of faces}=\text{number of edges}-\text{number of vertices}+2.
\]

Using this, we obtain that the number of faces of $\Lambda$ is exactly

\[ \left(\left(\displaystyle\sum_{a\in{B_{\Lambda}}}mult(a)\right)+\mathcal{V}\right)-\mathcal{V}+2=\left(\displaystyle\sum_{a\in{B_{\Lambda}}}mult(a)\right)+2.
\]

Of these, one is unbounded and the rest are bounded.
\end{proof}

Here are several examples.

\begin{example}
Let $f(z)=z^n$ for some $n\in\{1,2,3,\ldots\}$, and $\epsilon\in(0,\infty)$.  Then $E_{f,\epsilon}$ is the circle $\{z\in\mathbb{C}:|z|=\epsilon^{\frac{1}{n}}\}$.
\end{example}

\begin{example}
Let $f(z)=z^n-1$ for some $n\in\{2,3,\ldots\}$.  If $\epsilon\in(0,1)$, then $E_{f,\epsilon}$ has $n$ components, each a simple closed path which contains a single zero of $f$ in its bounded face.  If $\epsilon\in(1,\infty)$, then $E_{f,\epsilon}$ consists of a single simple closed path which contains all $n$ zeros of $f$ in its bounded face.  Finally, $E_{f,1}$ consists of a single component with a single vertex (at $0$), $n$ edges, and $n$ bounded faces, each of which contains a single zero of $f$.  In Figure~\ref{fig:p(z)=z^5-1levelcurves} below we see the example $f(z)=z^5-1$.  The level sets shown are the sets $E_{f,.5}$, $E_{f,1}$, and $E_{f,1.5}$.
\end{example}

\begin{figure}[H]
	\centering
		\includegraphics[width=.5\textwidth]{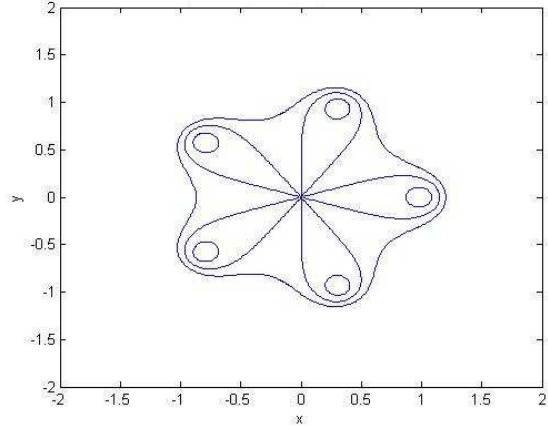}
	\caption{$f(z)=z^5-1$}
	\label{fig:p(z)=z^5-1levelcurves}
\end{figure}

As an application of the above results, we have a new proof of the following theorem of Gauss.  First a definition.

\begin{definition}
For $z\in{G}$, define $\Lambda_z$ to be the component of $E_{f,|f(z)|}$ which contains $z$.  If $z\in\partial{G}$, let $\Lambda_z$ denote the component of $\partial{G}$ which contains $z$.
\end{definition}

\begin{theorem}\label{Gauss' Theorem}
Let $p$ be a polynomial.  The critical points of $p$ are contained in the convex hull of the zeros of $p$.
\end{theorem}

\begin{proof}
Let $w_1,w_2,\ldots,w_n\in\mathbb{C}$ be the zeros of $p$ repeated according to multiplicity.  Assume by way of contradiction that there is some critical point $z_0$ of $p$ which is not in the convex hull of $\{w_i:1\leq{i}\leq{n}\}$.  By pre-composing with a linear map, we may assume that all zeros of $p$ are contained in the disk $\mathbb{D}=\{z\in\mathbb{C}:|z|<1\}$, and that $z_0\in(1,\infty)$.  Let $G$ denote one of the bounded faces of $\Lambda_{z_0}$ which is incident to $z_0$, and let $\gamma:[0,1]\to\mathbb{C}$ be a parameterization of the boundary of $G$.  Thus, $\gamma$ is a simple closed path with $\gamma(0)=\gamma(1)=z_0$.  The Maximum Modulus Theorem implies that $G$ contains a zero of $p$, and therefore $\partial{G}$ intersects the line $L\colonequals\{z\in\mathbb{C}:Re(z)=1\}$.  Define $r_1,r_2\in(0,1)$ by $r_1=\min(r\in[0,1]:\gamma(r)\in{L})$ and $r_2=\max(r\in[0,1]:\gamma(r)\in{L})$.

Then $\gamma(0,r_1)$ and $\gamma(r_2,1)$ are paths from $z_0$ to $L$ which do not intersect except at $z_0$ (and possibly in $L$).  Therefore since there are at least two bounded faces of $\Lambda_{z_0}$ which are incident to $z_0$, there are at least four paths in $\Lambda_{z_0}$ from $z_0$ to $L$ which do not intersect except at $z_0$ and in $L$.  It is not hard to see then that there is some $s\in(-1,1)\setminus\{0\}$ such that two of the paths intersect the set $\{z\in\mathbb{C}:Re(z)>1,Im(z)=s\}$.

\begin{figure}[H]
	\centering
		\includegraphics[width=.5\textwidth]{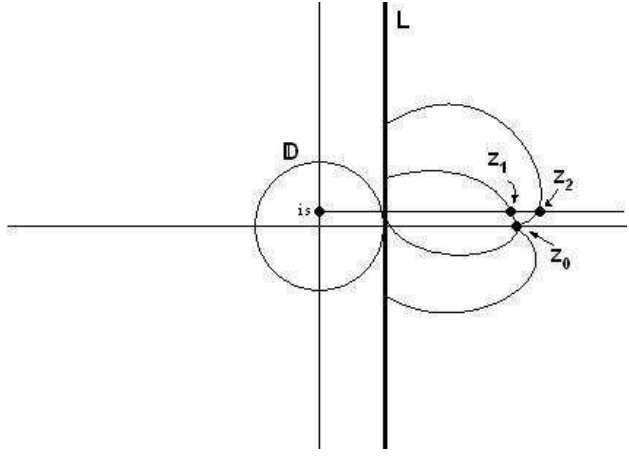}
	\caption{Gauss' Theorem}
	\label{fig:FigureForGaussTheorem}
\end{figure}

Let $z_1,z_2$ be distinct points in $\{z\in\mathbb{C}:Re(z)>1,Im(z)=s\}$ which are contained in $\Lambda_{z_0}$, with $Re(z_1)<Re(z_2)$, as in Figure~\ref{fig:FigureForGaussTheorem}.  Then for each zero $w_i$ of $p$, $|z_1-w_i|<|z_2-w_i|$.  Therefore

\[
|p(z_1)|=\displaystyle\prod_{i=1}^n|z_1-w_i|<\prod_{i=1}^n|z_2-w_i|=|p(z_2)|,
\]

which is a contradiction because $z_1$ and $z_2$ are in the same level curve of $p$.  Thus we conclude that there is no critical point $z_0$ of $p$ outside of the convex hull of the zeros of $p$.
\end{proof}

\section{Assorted Properties of Level Curves}\label{Assorted Properties of Level Curves}

\subsection{Setting}

We now impose the following additional restrictions on the set $G$ and the function $f$.

\begin{itemize}
	\item
	$|f|\equiv1$ on any unbounded component of $\partial{G}$.  For each bounded component $K$ of $\partial{G}$, there is some $\rho\in(0,\infty)$ such that $|f|\equiv\rho$ on $K$.  (That is, each component of $\partial{G}$ is contained in some level curve of $f$ in $G'$.)
	\begin{note}
If one of the components of $G^c$ is a single point $\{z\}$ for some $z\in\mathbb{C}$ (and thus $z$ is a removable discontinuity of $f$ restricted to $G$), then we replace $G$ with $G\cup\{z\}$.
	\end{note}

	\item
	If $G$ is unbounded, then $\displaystyle\lim_{z\in{G},z\to\infty}|f(z)|=1$.
	
	\item
	$G$ does not contain any unbounded level curve of $f$.

	\item
	$G$ is connected and $G^c$ has finitely many bounded components.

	\item
	No level curve of $f$ in $G$ intersects $\partial{G}$.  That is, for any $z\in{G}$, $\text{cl}(\lambda_z)\subset{G}$.
\end{itemize}

\begin{note}
Of particular interest to us for future work is the case $f=\frac{B_1}{B_2}$ where $B_1$ and $B_2$ are finite Blaschke products with $deg(B_1)\neq{deg(B_2)}$, with $G=\mathbb{D}$.  Note that when $B_1$ and $B_2$ are finite degree Blaschke products of the same degree, $f=\frac{B_1}{B_2}$ and $G=\mathbb{D}$ violates the final restriction above.  For example, if we set $B_1(z)=\dfrac{z-\frac{1}{2}}{1-\frac{1}{2}z}$ and $B_2(z)=\dfrac{z+\frac{1}{2}}{1+\frac{1}{2}z}$, and $f=\frac{B_1}{B_2}$ and $G=\mathbb{D}$, this choice of $f$ and $G$ satisfy the first four restrictions above, but the level curve of $f$ which contains $0$ consists of the imaginary axis along with the unit circle, and thus the final restriction is violated.
\end{note}

\subsection{Properties}

In the previous section we examined one level curve of $f$ in isolation from the others.  In this section we catalogue some facts about the level curves of $f$ globally in $G$.

The restriction $\displaystyle\lim_{z\in{G},z\to\infty}|f(z)|=1$ immediately implies that there are only finitely many zeros and poles of $f$ contained in $G$.  Our first result in this section extends this to the level curves of $f$ associated to values between $0$ and $\infty$.  We also show that there are only finitely many level curves of $f$ in $G$ which contain critical points of $f$.  We will first introduce some vocabulary and define a partial ordering on the set of all level curves of $f$ in $G$ which we will use in the proof.

\begin{definition}
A level curve of $f$ which is not a zero or pole of $f$, and which contains a critical point of $f$, we will call a critical level curve of $f$.  A level curve of $f$ which is not a critical level curve we call a non-critical level curve of $f$.
\end{definition}

\begin{definition}
If $\Lambda_1$ and $\Lambda_2$ are level curves of $f$ in $G$ then we say $\Lambda_1\prec\Lambda_2$ if $\Lambda_1$ lies in one of the bounded faces of $\Lambda_2$.  Let $D$ be an open sub-set of $G$.  If $\Lambda_1$ is a critical level curve of $f$ contained in $D$, then we say that $\Lambda_1$ is $\prec$-maximal with respect to $G$ if there is no other critical level curve $\Lambda_2$ of $f$ contained in $D$ such that $\Lambda_1\prec\Lambda_2$.
\end{definition}

\begin{proposition}\label{G has Only Finitely Many Crit. Level Curves.}
$G$ contains only finitely many components of $E_{f,\epsilon}$, and only finitely many critical level curves.
\end{proposition}

\begin{proof}
Suppose by way of contradiction that $\{L_i\}_{i=1}^{\infty}$ is a sequence of distinct components of $E_{f,\epsilon}$ contained in $G$.  Then there is some subsequence $\{L_{i_n}\}_{n=1}^{\infty}$ such that one of the following occurs:

\begin{itemize}
	\item
	$L_{i_1}\prec{L_{i_2}}\prec{L_{i_3}}\prec\cdots$.

	\item
	$L_{i_1}\succ{L_{i_2}}\succ{L_{i_3}}\succ\cdots$.

	\item
	For each $p,q\in\{1,2,\ldots\}$ with $p\neq{q}$, $L_{i_p}$ is contained in the unbounded face of ${L_{i_q}}$.
\end{itemize}

\begin{case}
$L_{i_1}\prec{L_{i_2}}\prec{L_{i_3}}\prec\cdots$.
\end{case}

For each $j\in\{1,2,\ldots,\}$, let $D_j$ denote the face of $L_{i_{j+1}}$ which contains $L_{i_j}$.  Let $\widetilde{D_j}$ denote the subset of $D_j$ which is exterior to $L_{i_j}$.  With this definition, we have a sequence of bounded disjoint open connected sets $\left\{\widetilde{D_j}\right\}_{j=1}^\infty$, such that $\partial{\widetilde{D_j}}\subset{G}$ for each $j$.  (One can use Lemma~\ref{Closed Sets Topological Lemma} to show that each $\widetilde{D_j}$ is connected.)  Moreover, since there are only finitely many bounded components of $G^c$, all but finitely many of the $\left\{\widetilde{D_j}\right\}_{j=1}^\infty$ are contained in $G$.

Thus (by dropping finitely many of the original sequence) we may assume that each of $\left\{\widetilde{D_j}\right\}_{j=1}^{\infty}$ is contained entirely in $G$.  However since $f$ is meromorphic on $G$, and $|f|\equiv\epsilon$ on $\partial{\widetilde{D_j}}$ for each $j\in\{1,2,\ldots\}$, the Maximum Modulus Theorem implies that $f$ has a zero or pole in each $\widetilde{D_j}$.  Thus we have a contradiction because as noted earlier, $f$ has only finitely many zeros and poles in $G$.

\begin{case}
$L_{i_1}\succ{L_{i_2}}\succ{L_{i_3}}\cdots$.
\end{case}

The exact same argument works as in the previous case with the appropriate changes.

\begin{case}
For each $p,q\in\{1,2,\ldots\}$ with $p\neq{q}$, $L_{i_p}$ is contained in the unbounded face of ${L_{i_q}}$.
\end{case}

For each $j\in\{1,2,\ldots\}$, let $D_j$ denote one of the bounded components of ${L_{i_j}}^c$.  Again, by dropping finitely many of the original $\{D_j\}_{j=1}^{\infty}$, we may assume that each of $\{D_j\}_{j=1}^{\infty}$ is contained in $G$, and thus must contain a zero or pole of $f$, which gives us the desired contradiction.

Thus by cases we have a contradiction of the original assumption, so we conclude that $G$ contains only finitely many level curves of $E_{f,\epsilon}$.

The fact that $G$ contains only finitely many critical level curves follows from an almost identical argument, letting $\{L_i\}_{i=1}^{\infty}$ be a sequence of critical level curves, and again obtaining a contradiction by using the fact that there are only finitely many zeros and poles of $f$ in $G$ and only finitely many components of $G^c$.
\end{proof}

In the next result we show that a closed set contained in the complement of a level curve of $f$ may be separated from that level curve by a different non-critical level curve of $f$.  This begins to show that the level sets of $f$ vary continuously.  First a topological fact that we will use several times throughout this paper, the proof of which follows directly from the compactness of $\bar{\mathbb{C}}$.

\begin{fact}\label{Closed Set Can Be Connected.}
Let $D\subset\bar{\mathbb{C}}$ be connected and open (either in $\mathbb{C}$ or in $\bar{\mathbb{C}}$).  Let $K\subset{D}$ be closed (in $\mathbb{C}$ or in $\bar{\mathbb{C}}$).  Then there is a closed connected set $\widetilde{K}\subset{D}$ which contains $K$.  Furthermore if $K$ is bounded, then $\widetilde{K}$ may be chosen to be bounded.
\end{fact}

\begin{definition}
For $D\subset\mathbb{C}$, let $sc(D)$ denote the union of $D$ with each bounded component of $D^c$.  (The "sc" stands for "simply connected.")
\end{definition}

\begin{proposition}\label{Level Curves are Separated from Compact Sets by Another Level Curve}
Let $L$ denote some level curve of $f$ in $G$, or some component of $\partial{G}$, and let $\eta\in[0,\infty]$ be the value that $|f|$ takes on $L$.  Let $K\subset{sc(G)}$ denote some closed set contained in a single bounded component of $L^c$, such that at least one of $L$ or $K$ is bounded.  If $z\in{G}$ is sufficiently close to $L$ and in the same component of $L^c$ as $K$, the following holds.

\begin{itemize}
	\item
	$|f(z)|\neq\eta$.

	\item
	$\Lambda_z$ is a non-critical level curve of $f$ in $G$.

	\item
	$K$ and $L$ are in different faces of $\Lambda_z$.
	
	\item
	If $L$ is bounded and $K$ is in the unbounded face of $L$, then $K$ is in the unbounded face of $\Lambda_z$.  Otherwise $K$ is in the bounded face of $\Lambda_z$.
\end{itemize}
\end{proposition}

\begin{proof}
Let $D$ denote the intersection of the component of $L^c$ which contains $K$ with $sc(G)$.

\begin{case}
$L$ is unbounded.
\end{case}

Since $L$ is unbounded, $K$ must be bounded.  Let $K_1\subset{D}$ be the set of all points $z\in{D}$ such that $z$ satisfies one of the following.

\begin{itemize}
	\item
	$z\in{K}$.

	\item
	$z\in{G}$ and $f(z)=0$ or $f(z)=\infty$.

	\item
	$z\in{G}$ and $|f(z)|=\eta$.

	\item
	$z\in{G}$ and $f'(z)=0$.

	\item
	$z\in{G}^c$.
\end{itemize}

$K_1$ is closed (as it is a finite union of closed sets), and by Proposition~\ref{G has Only Finitely Many Crit. Level Curves.}, $K_1$ is bounded and thus compact.  Therefore by Fact~\ref{Closed Set Can Be Connected.}, we may find some compact connected set $\widetilde{K_1}\subset{D}$ which contains $K_1$.  We can then find a compact connected set $K_2\subset{D}$ such that $\widetilde{K_1}$ is contained in the interior of $K_2$ (for example a finite union of closed balls contained in $D$).  Since $\widetilde{K_1}$ does not intersect $\partial{K_2}$, $|f|$ does not take the value $\eta$ on $\partial{K_2}$.  Thus since $\partial{K_2}$ is compact, if we set $\iota\colonequals\displaystyle\inf_{z\in\partial{K_2}}(||f(z)|-\eta|)$, then $\iota>0$.  We will now use this $\iota$ to determine how close the "sufficiently close" from the statement of the proposition is.

\begin{claim}
Some $\delta>0$ may be found small enough so that if $z\in{G}$ is within $\delta$ of $L$, then $z\notin{K_2}$, and $||f(z)|-\eta|\in(0,\iota)$.
\end{claim}

Since each level curve of $f$ in $G$ is bounded, $L$ must be an unbounded component of $\partial{G}$, and thus $\eta=1$.  However we will keep writing "$\eta$", as we will be using essentially the same argument for the next case, when $\eta$ may not equal $1$.  Since $K_2$ is compact, and $L$ is closed, $\delta>0$ may be chosen so that if $z\in{G}$ is within $\delta$ of $L$, then $z\notin{K_2}$.  Since $\displaystyle\lim_{z\in{G},z\to\infty}|f(z)|=\eta$, there is some $R>0$ such that if $z\in{G}$ with $|z|>R$, then $||f(z)|-\eta|<\iota$.  Furthermore, Proposition~\ref{G has Only Finitely Many Crit. Level Curves.} implies that the set $\{z\in{G}:|f(z)|=\eta\}$ is bounded, so by choosing $R$ larger if necessary, we may conclude that for $z\in{G}$ with $|z|>R$, $||f(z)|-\eta|\in(0,\iota)$.   On the other hand, by the compactness of $\{z\in{L}:|z|\leq{R}\}$, we can find some $\delta>0$ smaller than before such that if $z\in{G}$ is within $\delta$ of $L$, and $|z|\leq{R}$, then $||f(z)-\eta|\in(0,\iota)$, which gives us the result of the claim.

\begin{claim}
With $\delta>0$ chosen as in the previous claim, if $z\in{D}$ is less than $\delta$ away from $L$, then $\Lambda_z$ is a non-critical level curve of $f$ contained in $G$, such that $L$ is in the unbounded face of $\Lambda_z$, and $K_2$ is in the bounded face of $\Lambda_z$, and $|f(z)|\neq\eta$.
\end{claim}

Let $z\in{D}$ be less than $\delta$ away from $L$.  By the previous claim, $||f(z)|-\eta|\in(0,\iota)$ (and thus $|f(z)|\neq\eta$), so by definition of $\iota$, $\Lambda_z$ does not intersect $\partial{K_2}$.  Since $\Lambda_z$ does intersect ${K_2}^c$ (namely at $z$), and $\Lambda_z$ is connected, we may conclude that $\Lambda_z$ does not intersect $K_2$.  And $K_2$ is connected, so $K_2$ is entirely contained in one of the faces of $\Lambda_z$.  Since $K_2$ contains all critical points of $f$ in $D$, $\Lambda_z$ is a non-critical level curve of $f$, and therefore has only one bounded face.  By the Maximum Modulus Theorem this bounded face must contain either a zero or pole of $f$, or a point in $G^c$.  But each zero and pole of $f$ and point in $G^c$ which is in $D$ is contained in $K_2$, so we may conclude that $K_2$ is contained in the bounded face of $\Lambda_z$.  Since $L$ is unbounded, $L$ is contained in the unbounded face of $\Lambda_z$.

\begin{case}
$L$ is bounded.
\end{case}

If $K$ is contained in a bounded component of $L^c$, then the same argument as above works with minor changes, so let us assume that $K$ is contained in the unbounded component of $L^c$.  We use a similar argument, but we want to construct $K_2$ in a way that will guarantee that when we have found our $\delta$, if $z\in{D}$ is within $\delta$ of $L$, then $K_2$ will be contained in the unbounded face of $\Lambda_z$.  Therefore we define $K_1$ identically as before except in addition we include $\infty$ in $K_1$.  Constructing $K_2$ as before, we obtain a closed connected set contained in $D$ which contains $K_1$ in its interior.  Let $\iota$ and $\delta$ be defined as above.  The same argument as above with minor changes allows us to conclude if $z\in{D}$ is within $\delta$ of $L$, then $\Lambda_z$ is a non-critical level curve of $f$ in $G$ with $|f|\neq\eta$ on $\Lambda_z$, and $L$ is in the bounded face of $\Lambda_z$, and $K_2$ (and therefore $K$) is in the unbounded face of $\Lambda_z$.
\end{proof}

This last proposition hints that in some sense the components of $E_{f,\eta}$ vary continuously as $\eta$ varies.  And indeed this proposition does most of the work of showing this continuous variance, as it is used extensively in the proof of Proposition~\ref{Level Curves Vary Continuously}, which makes very explicit the continuous variance.  First two corollaries.

\begin{definition}
For $K\subset{G}$, define $\Lambda_K\colonequals\displaystyle\bigcup_{z\in{K}}\Lambda_z$.
\end{definition}

\begin{corollary}\label{If K is Compact then Lambda_K is Compact.}
Let $K\subset{G}$ be compact.  Then $\Lambda_K$ is compact.
\end{corollary}

\begin{proof}
We will first show that $\Lambda_K$ is closed in $\mathbb{C}$.  Let $z_0\in{\Lambda_K}^c$ be given.  We will show that there is some $\delta>0$ such that $B_{\delta}(z_0)\subset{\Lambda_K}^c$.

\begin{case}
$z_0\in{cl(G)}$.
\end{case}

Since $z_0\notin{\Lambda_K}$, $\Lambda_{z_0}\cap{K}=\emptyset$.  Thus by Proposition~\ref{Level Curves are Separated from Compact Sets by Another Level Curve}, there is some non-critical level curve $L$ of $f$ in $G$ such that $\Lambda_{z_0}$ is contained in one face of $L$ and $K$ is contained in the other face of $\Lambda_{z_0}$.  Let $D_1$ denote the face of $L$ which contains $z_0$ and let $D_2$ denote the face of $L$ which contains $K$.  For each $z\in{K}$, $z\in{D_2}$ and $\Lambda_z$ is connected and does not intersect $L$, so $\Lambda_z$ is contained in $D_2$.  Therefore $\Lambda_K\subset{D_2}$.  If we choose $\delta>0$ small enough that $B_{\delta}(z_0)\subset{D_1}$, then $B_{\delta}(z_0)\subset{\Lambda_K}^c$.

\begin{case}
$z_0\notin{cl(G)}$.
\end{case}

Since $z_0\notin{cl(G)}$, there is some $\delta>0$ such that $B_{\delta}(z_0)\subset{G^c}$, and thus $B_{\delta}(z_0)\subset{\Lambda_K}^c$.

We conclude that ${\Lambda_K}^c$ is open in $\mathbb{C}$, and thus $\Lambda_K$ is closed in $\mathbb{C}$.

Now we wish to show that $\Lambda_K$ is bounded.  Of course if $G$ is bounded then we are done, so suppose not.  Let $K_1$ denote the union of $K$ with all components of $E_{f,1}$ contained in $G$.  By Proposition~\ref{G has Only Finitely Many Crit. Level Curves.}, there are only finitely many components of $E_{f,1}$ contained in $G$, each of which is compact, so $K_1$ is compact.  Since $K_1$ is compact, there is a sequence of balls $B_1,B_2,\ldots,B_N$ each contained in $G$ such that $K_1\subset\displaystyle\bigcup_{i=1}^NB_i$.  Define $K_2\colonequals\displaystyle{cl\left(\bigcup_{i=1}^NB_i\right)}$.  $K_2$ is compact, so $\partial{K_2}$ is compact.  Therefore we may define $\rho\colonequals\displaystyle\min_{z\in\partial{K_2}}(||f(z)|-1|)$.  Since $\{z\in{G}:|f(z)|=1\}$ is contained in the interior of $K_2$, there is no point on the boundary of $K_2$ at which $|f|$ takes the value $1$.  Therefore we conclude that $\rho>0$.  We will show that $\Lambda_{K_2}$ is bounded, and thus $\Lambda_K$ is bounded as well.

We now break $\Lambda_{K_2}$ into two different sets, each of which we will show is bounded.  Define $F_1\colonequals\{z\in{K_2}:|f(z)|\in(1-\rho,1+\rho)\}$ and $F_2\colonequals\{z\in{K_2}:|f(z)|\notin(1-\rho,1+\rho)\}$.  $K_2=F_1\cup{F_2}$, so $\Lambda_{K_2}=\Lambda_{F_1}\cup\Lambda_{F_2}$.  For any $z\in{F_1}$, $z$ is contained in the interior of $K_2$, and $\Lambda_z$ does not intersect $\partial{K_2}$, so $\Lambda_z\subset{K_2}$.  Therefore $\Lambda_{F_1}\subset{K_2}$, so $\Lambda_{F_1}$ is bounded.  On the other hand, for each $z\in\Lambda_{F_2}$, $|f(z)|\notin(1-\rho,1+\rho)$, and $\displaystyle\lim_{w\in{G},w\to\infty}|f(w)|=1$, so $\Lambda_{F_2}$ is bounded.

Therefore we conclude that $\Lambda_{K_2}$ is bounded, and $K\subset{K_2}$, so $\Lambda_{K}$ is bounded.  Finally we conclude that $\Lambda_K$ is compact.
\end{proof}

\begin{definition}
If $D$ is a face of $\Lambda$, say that $f$ is increasing into $D$ if there is some $\iota>0$ such that $|f|>\epsilon$ on $\{z\in{D}:d(\{z\},\Lambda)<\iota\}$, and say that $f$ is decreasing into $D$ if there is some $\iota>0$ such that $|f|<\epsilon$ on $\{z\in{D}:d(\{z\},\Lambda)<\iota\}$.
\end{definition}

\begin{corollary}\label{f Increases or Decreases into Face of Lambda.}
Let $D$ be a face of $\Lambda$.  Then either $f$ is increasing into $D$ or $f$ is decreasing into $D$.
\end{corollary}

\begin{proof}
We begin by assuming that $D$ is one of the bounded faces of $\Lambda$.

Define $A$ to be the union of the set $\{z\in{D\cap{G}}:|f(z)|=\epsilon\}$ with any components of $G^c$ contained in $D$.  If $A$ defined as such is empty, then let $A$ be just some single point in $D$.  $A$ is closed, so by Proposition~\ref{Level Curves are Separated from Compact Sets by Another Level Curve}, there is some non-critical level curve of $f$ contained in $D\cap{G}$ which separates $\Lambda$ from $A$.  Call this level curve $L$.  Let $\eta>0$ be such that $|f|\equiv\eta$ on $L$.  Since $L$ separates $A$ from $\Lambda$, $L$ does not intersect $A$, so $\eta\neq\epsilon$.

\begin{case}
$\eta>\epsilon$.
\end{case}

Let $D_1$ denote the face of $L$ which contains $\Lambda$, and let $D_2$ denote the face of $L$ which contains $A$.  Define $D'\colonequals{D_1\cap{D}}$, the portion of $D$ which is "between" $L$ and $\Lambda$.  $D\setminus{D_1}=cl(D_2)$, so $D\setminus{D_1}$ is compact and simply connected.  We may use Lemma~\ref{Closed Sets Topological Lemma} to show that $D\setminus{cl(D_2)}=D\cap{D_1}$ is connected.  Therefore since $|f|\neq\epsilon$ on $D\cap{D_1}$, and $|f|$ is continuous, either $|f|>\epsilon$ on $D\cap{D_1}$, or $|f|<\epsilon$ on $D\cap{D_1}$.  But there are points in $\partial(D\cap{D_1})$ at which $|f|$ takes values greater than $\epsilon$, namely any point in $L$, so there must be points in $D\cap{D_1}$ at which $|f|$ takes values greater than $\epsilon$ by the continuity of $|f|$.  Thus $|f|>\epsilon$ on $D\cap{D_1}$.  Define $\iota\colonequals{d(L,\Lambda)}$.  If $w\in\{z\in{D}:d(\{z\},\Lambda)<\iota\}$, then $w\in{D\cap{D_1}}$, so $|f(w)|>\epsilon$.  Thus $f$ is increasing into $D$.

\begin{case}
$\eta<\epsilon$.
\end{case}

The same argument as above works, with the conclusion that $f$ is decreasing into $D$.

If $D$ is the unbounded face of $\Lambda$, the same argument works with the appropriate minor changes.
\end{proof}

Now we finally make explicit the notion that the components of $E_{f,\eta}$  "vary continuously in $\eta$".  Recall that the Hausdorff distance $\check{d}$ (defined below) may be characterized by saying that for two sets $X,Y\subset\mathbb{C}$ to be close with respect $\check{d}$, each $x\in{X}$ must be close to $Y$, and each $y\in{Y}$ must be close to $X$.

\begin{definition}
For $X,Y\subset\mathbb{C}$, define $\displaystyle{d}_1(X,Y)\colonequals\sup_{x\in X}(d(\{x\},Y))$, and define $\displaystyle{d}_2(X,Y)\colonequals\sup_{y\in Y}(d(X,\{y\}))$.  Then the distance function we are interested in is $\check{d}(X,Y)\colonequals\max(d_1(X,Y),d_2(X,Y))$.  If either $X$ or $Y$ are empty, we define $\check{d}(X,Y)\colonequals\infty$.
\end{definition}

The following proposition states that if $\eta>0$ is sufficiently close to $\epsilon$, then some collection of level curves from $E_{f,\eta}$ is close to $\Lambda$ with respect to $\check{d}$.  First we will list a fact to be used in the proof.

\begin{fact}\label{Finite Set in D Approximates Boundary of D.}
Let $D\subset\mathbb{C}$ be a non-empty set such that $\partial{D}$ is bounded.  Then if $\delta>0$ is given, there is some finite subset $A\subset{D}$ such that $\check{d}(A,\partial{D})<\delta$.
\end{fact}

This follows from the compactness of $\partial{D}$.

\begin{proposition}\label{Level Curves Vary Continuously}
Let $\delta>0$ be given.  Then there is some $\eta\in(0,\epsilon)$ such that for each $\zeta\in(\epsilon-\eta,\epsilon+\eta)$, there is some collection $L_1,\ldots,L_N$ of level curves of $f$ contained in $G$ such that $|f|\equiv\zeta$ on each $L_i$, and $\check{d}\left(\displaystyle\left(\bigcup_{i=1}^NL_i\right),\Lambda\right)<\delta$.
\end{proposition}

\begin{proof}
Let $\delta>0$ be given.  Let $E_1,E_2,\ldots,E_N$ be an enumeration of the edges of $\Lambda$, and for each $i\in\{1,2,\ldots,N\}$, let $z_i$ be some fixed interior point of $E_i$ (that is, a point in $E_i$ which is not an endpoint of $E_i$).  Our first goal is to find an $\iota>0$ small enough so that each of the following hold:

\begin{enumerate}
	\item\label{item: 1.}
	For each $i\in\{1,2,\ldots,N\}$, the only faces of $\Lambda$ that intersect $B_{\iota}(z_i)$ are the two faces adjacent to $E_i$.

	\item\label{item: 2.}
	If $D$ is one of the faces of $\Lambda$, and $z\in{D}$ is less than $\iota$ away from $\Lambda$, then $\check{d}(\Lambda_z,\partial{D})<\frac{\delta}{2}$.

	\item\label{item: 3.}
	If $D$ is one of the faces of $\Lambda$, and we define $A\colonequals\{z\in{D}:d(\{z\},\Lambda)<\iota\}$, then either $|f|<\epsilon$ on $A$ or $|f|>\epsilon$ on $A$.
\end{enumerate}

Suppose an $\iota$ may be found which satisfies each of these three items.  Our next goal is to find some $\eta\in(0,\epsilon)$ small enough so that if $\zeta\in(\epsilon-\eta,\epsilon+\eta)$, then for each $i\in\{1,2,\ldots,N\}$, there is some point in ${B_{\iota}(z_i)}$ at which $|f|$ takes the value $\zeta$.  Suppose such an $\eta>0$ may be found.  We now show that the statement of the proposition holds for this $\delta$ and $\eta$.

Let $\zeta\in(\epsilon-\eta,\epsilon+\eta)$ be given.  Of course if $\zeta=\epsilon$, then by putting $N=1$ and $L_1=\Lambda$, the statement of the proposition obviously holds, so let us assume that $\zeta\neq\epsilon$.  For each $i\in\{1,2,\ldots,N\}$, let $w_i$ be a fixed point in $B_{\iota}(z_i)$ at which $|f|$ takes the value $\zeta$, which may be found by Item~\ref{item: 3.} above.  Define $\mathcal{L}\colonequals\displaystyle\bigcup_{i=1}^N\Lambda_{w_i}$.

\begin{claim}
$\check{d}(\mathcal{L},\Lambda)<\delta$.
\end{claim}

Let $x\in\mathcal{L}$ be given.  Let $i\in\{1,2,\ldots,N\}$ be such that $x\in\Lambda_{w_i}$.  Since $d(\{w_i\},\Lambda)<\iota$, we have that $\check{d}(\Lambda_{w_i},\Lambda)<\frac{\delta}{2}$ by Item~\ref{item: 2.}, and thus $d(\{x\},\Lambda)<\frac{\delta}{2}$.  Therefore $d_1(\mathcal{L},\Lambda)<\delta$.

Let $x\in\Lambda$ be given.  Let $i\in\{1,2,\ldots,N\}$ be such that $x\in{E_i}$.  Let $D$ denote the face of $\Lambda$ which contains $w_i$.  $d(\{z_i\},\{w_i\})<\iota$, so by Item~\ref{item: 1.} above, $E_i\subset\partial{D}$.  And $d(\{w_i\},\Lambda)<\iota$, so by Item~\ref{item: 2.} above, $\check{d}(\Lambda_{w_i},\partial{D})<\frac{\delta}{2}$.  Since $x\in{E_i}\subset\partial{D}$, we conclude that $d(\Lambda_{w_i},\{x\})<\frac{\delta}{2}$.  Since $x\in\Lambda$ was chosen arbitrarily, we conclude that $d_2(\mathcal{L},\Lambda)<\delta$.

Finally we conclude that $\check{d}(\mathcal{L},\Lambda)<\delta$, and we are done, subject to finding the specified $\iota$ and $\eta$.

Of course if we show that a positive constant may be chosen to satisfy each of the three above items individually, then the minimum of the three constants will have the properties desired in a choice of $\iota$.

We first show that $\iota$ may be chosen to satisfy Item~\ref{item: 1.}.  Let $i\in\{1,2,\ldots,N\}$ be given.  Define ${E_i}'$ to be the points in $E_i$ which are not endpoints of $E_i$.  Then $\Lambda\setminus{E_i}'$ is closed, so $r_i\colonequals{d(\{z_i\},\Lambda\setminus{E_i}')}>0$.  Now if $D$ is a face of $\Lambda$ such that $z_i\notin\partial{D}$ (and thus $\partial{D}\subset\Lambda\setminus{E_i}'$), and $w\in{D}$, then any path from $w$ to $z_i$ intersects $\partial{D}$, in particular the straight line path.  Therefore

\[
d(\{z_i\},\{w\})\geq{d(\{z_i\},\partial{D})}\geq{d(\{z_i\},\Lambda\setminus{E_i}')}=r_i,
\]

and therefore we conclude that $D$ does not intersect $B_{r_i}(z_i)$.  If we choose $\iota>0$ smaller than each $r_i$, this $\iota$ satisfies Item~\ref{item: 1.}.

We now show that $\iota$ may be chosen to satisfy Item~\ref{item: 2.}.  Let $D$ be one of the faces of $\Lambda$.  By Fact~\ref{Finite Set in D Approximates Boundary of D.}, there is a finite sequence of points $x_1,x_2,\ldots,x_k\in{D}$ such that $\displaystyle\check{d}\left(\bigcup_{i=1}^k\{x_i\},\partial{D}\right)<\frac{\delta}{4}$.  Define

\[
K\colonequals\{x_i\}_{i=1}^k\cup\left\{z\in{D}:d(\{z\},\partial{D})\geq\frac{\delta}{4}\right\}.
\]

By Proposition~\ref{Level Curves are Separated from Compact Sets by Another Level Curve}, there is some $\iota_D>0$ such that if $z\in{D}$ is less than $\iota_D$ away from $\Lambda$, then $\Lambda_z$ contains $\Lambda$ in one face and $K$ in the other.  Fix some $w\in{D}$ less than $\iota_D$ away from $\Lambda$.

\begin{claim}
$\check{d}(\Lambda_w,\partial{D})<\frac{\delta}{2}$.
\end{claim}

Let $w_0\in\Lambda_w$ be given.  Since $\Lambda_w$ is contained in $D\setminus{K}$, we have $\Lambda_w\subset\left\{z\in{D}:d(\{z\},\partial{D})<\frac{\delta}{4}\right\}$.  Thus $d(\{w_0\},\partial{D})<\frac{\delta}{4}$, and thus $d_1(\Lambda_w,\partial{D})<\frac{\delta}{2}$.

Let $z_0\in\partial{D}$ be given.  Then for some $i\in\{1,2,\ldots,k\}$, $d(\{x_i\},\{z_0\})<\frac{\delta}{4}$.  But $x_i\in{K}$, so $x_i$ and $z_0$ are in different faces of $\Lambda_w$.  Since $x_i$ and $z_0$ are in different faces of $\Lambda_w$, the straight line path from $z_0$ to $x_i$ intersects $\Lambda_w$, and thus $d(\Lambda_w,\{z_0\})\leq{d(\{x_i\},\{z_0\})}<\frac{\delta}{4}$.  Thus $d_2(\Lambda_w,\partial{D})<\frac{\delta}{2}$.

Therefore we conclude that $\check{d}(\Lambda_w,\partial{D})<\frac{\delta}{2}$.  Define $\iota\colonequals\displaystyle\min(\iota_D:D\text{ is a face of }\Lambda)$.  Then this $\iota$ satisfies Item~\ref{item: 2.}.

Finally the fact that $\iota$ may be chosen to satisfy Item~\ref{item: 3.} follows directly from Corollary~\ref{f Increases or Decreases into Face of Lambda.}.

The fact that an $\eta>0$ may be found with the desired property follows directly from the Open Mapping Theorem.  Since the desired $\iota>0$ and $\eta>0$ may be found, we are done.
\end{proof}

The next result is one of the most important results of the paper.  This theorem states that if any two level curves of $f$ in $G$ are exterior to each other, then there is a critical level curve of $f$ in $G$ which contains the two level curves in different bounded faces.  We will conclude from this that if we remove all critical level curves from $G$, each component of the remaining set will be conformally equivalent to a disk or an annulus.  If, in addition, we remove the zeros and poles of $f$ from $G$, then each component of the remaining set will be conformally equivalent to an annulus.

\begin{theorem}\label{Two Level Curve Implies Crit. Level Curve.}
Let each of $L_1$ and $L_2$ be a level curve of $f$ contained in $G$ or a bounded components of $G^c$.  Let $F_1$ denote the unbounded face of $L_1$ and $F_2$ the unbounded face of $L_2$.  If $L_1\subset{F_2}$, and $L_2\subset{F_1}$, then there is some $w\in{G}$ which lies in $F_1\cap{F_2}$, such that $f'(w)=0$ and $L_1$ and $L_2$ are contained in different bounded faces of $\Lambda_w$.
\end{theorem}

\begin{proof}
Define $\epsilon_1\colonequals|f(L_1)|$, and $\epsilon_2\colonequals|f(L_2)|$.

\begin{case}\label{L_1 and L_2 are Level Curves of f in G.}
Both $L_1$ and $L_2$ are level curves of $f$ in $G$.
\end{case}

$G\cap{F_1}\cap{F_2}$ is open, and by Lemma~\ref{Closed Sets Topological Lemma}, it is connected.  From this it is easy to show that we can find a path $\gamma:[0,1]\to{G}$ such that $\gamma(0)\in{L_1}$ and $\gamma(1)\in{L_2}$, and for all $r\in(0,1)$, $\gamma(r)\in{G}\cap{F_1\cap{F_2}}$.  For $r\in(0,1)$, define $K_r\colonequals\Lambda_{\gamma(r)}$, and define $A\subset(0,1)$ be the set such that $r\in{A}$ if and only if $K_r$ contains $L_1$ in one of its bounded faces $L_2$ in its unbounded face.  Clearly if $L$ is any level curve of $f$ in $G$ such that $L_1$ and $L_2$ are in different faces of $L$, then $L$ intersects the path $\gamma$.  Thus Proposition~\ref{Level Curves are Separated from Compact Sets by Another Level Curve} guarantees that $A$ is non-empty.  So if we define $r_1\colonequals\sup(r\in(0,1):r\in{A})$, we have $r_1\in(0,1]$.

\begin{claim}
$r_1<1$.
\end{claim}

Proposition~\ref{Level Curves are Separated from Compact Sets by Another Level Curve} also implies that we may find some $s\in(0,1)$ such that $K_s$ contains $L_2$ in one of its bounded faces and $L_1$ in its unbounded face.  Let $D$ denote the face of $K_s$ which contains $L_2$.  Since $K_s$ and $L_2$ are compact, the distance between $K_s$ and $L_2$ is greater than zero, so because $\gamma$ is continuous, there is some $\eta>0$ such that if $r\in(1-\eta,1]$, $\gamma(r)$ is contained in $D$.  Therefore for all $r\in(1-\eta,1)$, $K_r$ is contained in $D$, so the bounded faces of $K_r$ are contained in $D$.  Thus for all $r\in(1-\eta,1)$, $r$ is not in $A$.  From this we may conclude that $r_1<1$.

We will now show that $K_{r_1}$ contains a critical point of $f$, and that $L_1$ and $L_2$ are contained in different bounded faces of $K_{r_1}$.

\begin{claim}\label{L_1 prec K_{r_1}}
$L_1\prec{K_{r_1}}$.
\end{claim}

Since $\gamma(r_1)\in{F_1\cap{F_2}}$, $K_{r_1}\subset{F_1\cap{F_2}}$.  Suppose by way of contradiction that $L_1$ is contained in the unbounded face of $K_{r_1}$.  Then by Proposition~\ref{Level Curves are Separated from Compact Sets by Another Level Curve}, there is some non-critical level curve $\Lambda_1$ of $f$ contained in $G$ such that $K_{r_1}$ is contained in the bounded face of $\Lambda_1$ and $L_1$ is contained in the unbounded face of $\Lambda_1$.  Then one may easily use the continuity of $\gamma$ to show there is some $s\in(0,r_1)$ such that for each $r\in(s,r_1]$, $K_r$ does not contain $L_1$ in any of its bounded faces (since $\gamma(r)$ is in the bounded face of $L_1$).  Therefore for each $r\in(s,r_1]$, $r\notin{A}$, which is a contradiction of the definition of $r_1$.  Thus we conclude that $L_1\prec{K_{r_1}}$.

\begin{claim}
$L_2\prec{K_{r_1}}$.
\end{claim}

Suppose by way of contradiction that $L_2$ is in the unbounded face of $K_{r_1}$.  Then Proposition~\ref{Level Curves are Separated from Compact Sets by Another Level Curve} gives that there is some non-critical level curve $\Lambda_2$ of $f$ in $G$ such that $K_{r_1}$ in the bounded face of $\Lambda_2$, and $L_2$ in the unbounded face of $\Lambda_2$.  Thus $\gamma(r_1)$ is in the bounded face of $\Lambda_2$, and $\gamma(1)$ is in the unbounded face of $\Lambda_2$, so by the continuity of $\gamma$, there is some $r\in(r_1,1)$ such that $\gamma(r)\in\Lambda_2$ (and thus $K_r=\Lambda_2$).  But $L_1$ is in one of the bounded faces of $K_{r_1}$, so $L_1$ is in the bounded face of $\Lambda_2$ as well, and thus $r\in{A}$.  However this is a contradiction of the definition of $r_1$.  We conclude that $L_2$ is contained in one of the bounded faces of $K_{r_1}$.  That is, $L_2\prec{K_{r_1}}$.

Thus $L_1$ and $L_2$ are each contained in a bounded face of $K_{r_1}$.  We now wish to show that $L_1$ and $L_2$ are contained in different bounded faces of $K_{r_1}$.

\begin{claim}
$L_1$ and $L_2$ are contained in different bounded faces of $K_{r_1}$.
\end{claim}

We use an almost identical argument to that found in Claim~\ref{L_1 prec K_{r_1}}.  Suppose by way of contradiction that $L_1$ and $L_2$ are contained in the same bounded face of $K_{r_1}$.  Let $D$ denote the face of $K_{r_1}$ containing $L_1$ and $L_2$.  Again by Proposition~\ref{Level Curves are Separated from Compact Sets by Another Level Curve}, there is some non-critical level curve $\Lambda_3$ of $f$ contained in $D$ such that $L_1$ and $L_2$ are in the bounded face of $\Lambda_3$, and $K_{r_1}$ is in the unbounded face of $\Lambda_3$.  Therefore $\gamma(0)$ is in the bounded face of $\Lambda_3$ and $\gamma(r_1)$ is in the unbounded face of $\Lambda_3$.  Thus by the continuity of $\gamma$ there is some $s\in(0,r_1)$ such that for each $r\in(s,r_1]$, $\gamma(r)$ is in the unbounded face of $\Lambda_3$.  Fix for the moment some $r\in(s,r_1]$.  $K_r$ is contained in the unbounded face of $\Lambda_3$, and $L_1$ and $L_2$ are both contained in the bounded face of $\Lambda_3$.  Therefore since $K_r$ does not intersect $\Lambda_3$, $L_1$ is in the bounded face of $K_r$ if and only if $L_2$ is as well.  Thus $r\notin{A}$ for each $r\in(s,r_1]$, which contradicts the definition of $r_1$.

We conclude that $L_1$ and $L_2$ are contained in different bounded faces $K_{r_1}$, which implies that $K_{r_1}$ contains a zero of $f'$ by Corollary~\ref{Vertices are Crit. Points of f.}, and we are done.

\begin{case}
At least one of $L_1$ and $L_2$ are bounded components of $\partial{G}$.
\end{case}

Suppose that $L_1$ is a bounded component of $G^c$.  $L_2$ is compact, so by Proposition~\ref{Level Curves are Separated from Compact Sets by Another Level Curve}, we may find a non-critical level curve ${L_1}'$ of $f$ in $G$ such that $L_1$ is contained in the bounded face of ${L_1}'$ and $L_2$ is contained in the unbounded face of ${L_1}'$.  If $L_2$ is a bounded component of $G^c$, the same use of Proposition~\ref{Level Curves are Separated from Compact Sets by Another Level Curve} gives us a level curve ${L_2}'$ of $f$ in $G$ such that $L_2$ is contained in the bounded face of ${L_2}'$ and ${L_1}'$ (and thus $L_1$) is contained in the unbounded face of ${L_1}'$.  Then we can use the preceding case with ${L_1}'$ and ${L_2}'$, which then gives us the desired result for $L_1$ and $L_2$.
\end{proof}

Let us now extend the $\prec$-ordering to the bounded components of $\partial{G}$ as follows.

\begin{definition}
If $K\subset{G}$ is a bounded component of $\partial{G}$, and $L$ is a level curve of $f$ in $G$ such that $K$ is contained in one of the bounded faces of $L$, then we write $K\prec{L}$.
\end{definition}

Next we have two corollaries to the previous theorem guaranteeing the existence of a unique $\prec$-maximal critical level curve of $f$ in certain regions of $G$, but first a definition.

\begin{definition}
Define 

\[
\mathcal{B}\colonequals\{z\in{G}:f'(z)=0\text{ or }f(z)=0\text{ or }f(z)=\infty\},
\]

and define 

\[
\mathcal{C}\colonequals\displaystyle\left(\bigcup_{z\in\mathcal{B}}\Lambda_z\right)\cup\{z\in\partial{G}:z\text{ is in a bounded component of }\partial{G}\}.
\]

Recall that since $\displaystyle\lim_{z\in{G},z\to\infty}f(z)=1$, $f$ has only finitely many zeros and poles in $G$, and $G^c$ was stipulated to have only finitely many bounded components, so Proposition~\ref{G has Only Finitely Many Crit. Level Curves.} immediately implies that $\mathcal{B}$ is finite, and thus that $\mathcal{C}$ has finitely many components.
\end{definition}

\begin{corollary}\label{Existence of Prec Maximal Element in G}
There is a unique $\prec$-maximal component of $\mathcal{C}$.
\end{corollary}

\begin{proof}
By the Maximum Modulus Theorem, $\mathcal{C}$ is non-empty.  Since $\mathcal{C}$ has finitely many components, it must have at least one $\prec$-maximal component.  Suppose by way of contradiction that there are two distinct components $L_1$ and $L_2$ of $\mathcal{C}$ which are both $\prec$-maximal.  Then each is in the unbounded component of the other, so Theorem~\ref{Two Level Curve Implies Crit. Level Curve.} guarantees that there is some other critical level curve $\Lambda'\subset\mathcal{C}$ such that $L_i\prec\Lambda'$ for $i=1,2$.  But this is a contradiction of the $\prec$-maximality of $L_1$ and $L_2$.
\end{proof}

More generally we have the following corollary.

\begin{corollary}\label{Existence of Prec Maximal Element in D}
If $D$ is a bounded face of $\Lambda$ then there is a unique component of $\mathcal{C}$ which is $\prec$-maximal with respect to $D$.
\end{corollary}

\begin{proof}
Just replace $G$ with $D$, and then the desired result is just the result from Corollary~\ref{Existence of Prec Maximal Element in G} in the new setting.
\end{proof}

In light of our following final result, the name "critical level curve" has new meaning.  The following theorem says that on any component of $G\setminus\mathcal{C}$, $f$ is very simple, being conformally equivalent to the function $z\mapsto{z^n}$ for some $n\in\{1,2,\ldots\}$.  First a bit of notation.

\begin{definition}
For any path $\gamma$, let $-\gamma$ denote the same curve $\gamma$ equipped with the opposite orientation.
\end{definition}

\begin{theorem}\label{f Conformally Equiv. to a Power of z}
Let $D$ be a component of $G\setminus\mathcal{C}$.  Then the following hold.

\begin{enumerate}
\item
$D$ is conformally equivalent to some annulus $A$.

\item
Let $E_1$ denote the inner boundary of $D$, and let $E_2$ denote the outer boundary of $D$.  Then there is some $\epsilon_1,\epsilon_2\in[0,\infty]$ such that $\epsilon_1\neq\epsilon_2$, and $|f|\equiv\epsilon_1$ on $E_1$, and $|f|\equiv\epsilon_2$ on $E_2$.

\item\label{item:Description of phi.}  Let $i_1,i_2\in\{1,2\}$ be chosen so that $\epsilon_{i_1}<\epsilon_{i_2}$.  Then there is some $N\in\{1,2,\ldots\}$ such that $A=ann(0;{\epsilon_{i_1}}^{\frac{1}{N}},{\epsilon_{i_2}}^\frac{1}{N})$, and some conformal mapping $\phi:D\to{ann(0;\epsilon_1^{\frac{1}{N}},\epsilon_2^{\frac{1}{N}})}$ such that $f\equiv\phi^M$ on $D$, where $M=\pm{N}$.

\item
The conformal map $phi$ described in Item~\ref{item:Description of phi.} extends continuously to $E_2$ and to all points in $E_1$ which are not zeros of $f'$.  If $z\in{E_1}$ is a zero of $f'$, and $\gamma:[0,1]\to{G}$ is a path such that $\gamma([0,1))\subset{D}$, and $\gamma(1)=z$, then $\displaystyle\lim_{r\to1}\gamma(r)$ exists.
\end{enumerate}
\end{theorem}

\begin{proof}
Let $D$ be some component of $G\setminus\mathcal{C}$.  Since $D$ does not contain a zero or pole of $f$, the Maximum Modulus Theorem implies that $D^c$ must have at least one bounded component.  Suppose by way of contradiction that $D^c$ has two distinct bounded components.  Replacing $G$ with $D$ in the statement of Theorem~\ref{Two Level Curve Implies Crit. Level Curve.}, we may conclude that $D$ contains a zero of $f'$, which is a contradiction because all zeros of $f'$ in $G$ are contained in $\mathcal{C}$.  We conclude that $D^c$ has exactly one bounded component.  Thus $D$ is conformally equivalent to an annulus (see for example~\cite{C2}).  Let $E_1$ denote the interior boundary of $D$ and $E_2$ denote the exterior boundary of $D$.  Each component of the boundary of $D$ is contained in a level curve of $f$ or a component of $\partial{G}$.  Therefore we may define $\epsilon_1\in[0,\infty]$ to be the value of $|f|$ on $E_1$ and $\epsilon_2\in[0,\infty]$ to be the value of $|f|$ on $E_2$.  By the Maximum Modulus Theorem since $D$ does not contain a zero or pole of $f$ and $D\subset{G}$, we may conclude that $\epsilon_1\neq\epsilon_2$.  Assume throughout that $\epsilon_1<\epsilon_2$, otherwise make the appropriate minor changes.

\begin{claim}\label{Same Change in Arg(f) along any Level Curve in D}
There is some $N\in\{\pm1,\pm2,\ldots\}$ such that for any $z\in{D}$, the change in $\arg(f)$ as $\Lambda_z$ is traversed in the positive direction is exactly $2\pi{N}$.
\end{claim}

Let $z_1,z_2\in{D}$ be given.  Since $D$ contains no critical points of $f$, $\Lambda_{z_1}$ and $\Lambda_{z_2}$ are non-critical level curves of $f$, and Theorem~\ref{Two Level Curve Implies Crit. Level Curve.} implies that either $\Lambda_{z_1}\prec\Lambda_{z_2}$ or $\Lambda_{z_2}\prec\Lambda_{z_1}$.  Rename $z_1$ and $z_2$ if necessary so that $\Lambda_{z_1}\prec\Lambda_{z_2}$.  By the Maximum Modulus Theorem, the bounded face of $\Lambda_{z_1}$ contains some point in $\mathcal{C}$.  However $D$ contains no points of $\mathcal{C}$, so $E_1$ is a contained in the bounded face of $\Lambda_{z_1}$.  Let Figure~\ref{fig:D,z_1,z_2} represent this choice of $D$, $z_1$, and $z_2$.  Let $\sigma$ denote the path in Figure~\ref{fig:D with sigma}.

\begin{figure}[ht]
\begin{minipage}[b]{0.45\linewidth}
\centering
\includegraphics[width=\textwidth]{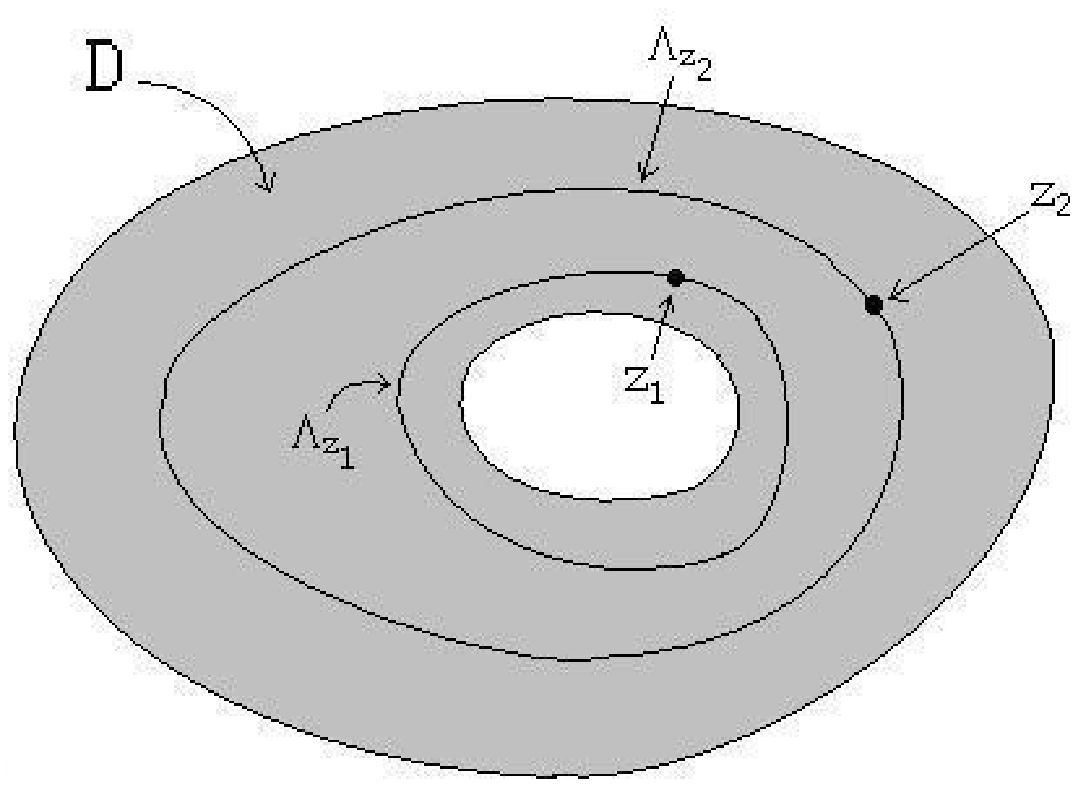}
\caption{$D$, $z_1$, and $z_2$.}
\label{fig:D,z_1,z_2}
\end{minipage}
\hspace{0.5cm}
\begin{minipage}[b]{0.45\linewidth}
\centering
\includegraphics[width=\textwidth]{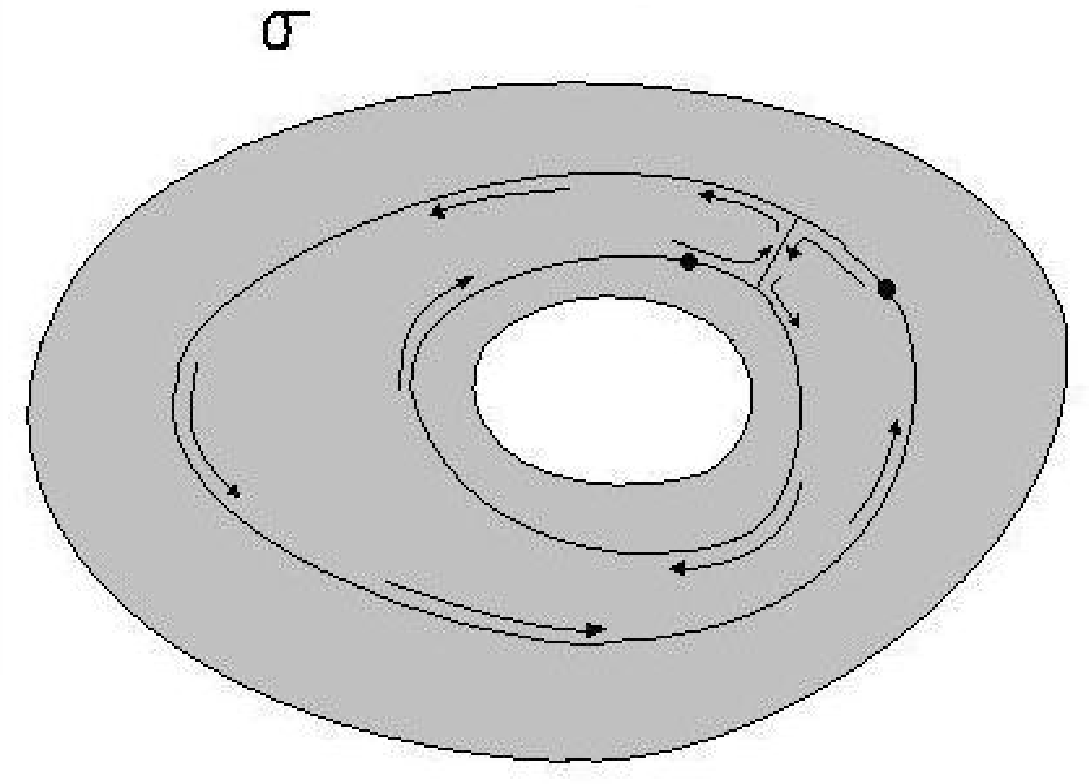}
\caption{Definition of $\sigma$.}
\label{fig:D with sigma}
\end{minipage}
\end{figure}

Since $\sigma$ may be retracted to a point in $D$, the total change in $\arg(f)$ along $\sigma$ is zero.  Therefore the sum of the changes in $\arg(f)$ along $\Lambda_{z_2}$ and $-\Lambda_{z_1}$ equals zero.  Therefore the change in $\arg(f)$ along $\Lambda_{z_1}$ is the same as the change in $\arg(f)$ along $\Lambda_{z_2}$.  We conclude that the change in $\arg(f)$ along $\Lambda_z$ is independent of the choice of $z\in{D}$.  Let $\beta$ denote this common number.  It is well known (see for example~\cite{C1}) that if $\gamma$ is any closed path in $\mathbb{C}$ and $g$ is a function analytic on $\gamma$ with no zeros on $\gamma$, the change in $\arg(g)$ along $\gamma$ is well defined and there is some $n\in\mathbb{Z}$ such that the change in $\arg(g)$ along $\gamma$ is $2\pi{n}$.  Thus we may choose $N\in\mathbb{Z}$ such that $\beta=2\pi{N}$.  By the same argument as above, the changes in $\arg(f)$ along $E_1$ and $E_2$ are both $2\pi{N}$ as well.

Furthermore since $D$ does not contain any zero of $f'$, $f$ is injective at each point in $D$, so if $L$ is some level curve of $f$ in $D$, $\arg(f)$ is either strictly increasing or strictly decreasing as $L$ is traversed in the positive direction.  Therefore $N\neq0$, so $N\in\{\pm1,\pm2,\ldots\}$.

Let us now adopt the convention that unless otherwise specified, any time we are referring to a level curve of $f$ in $D$ as a path, we are traversing that level curve in the direction in which $\arg(f)$ is increasing.  By a similar argument to the one in the claim above one may show that this is either the positive orientation for all level curves of $f$ in $D$ or the negative orientation for all level curves of $f$ in $D$.  Thus with this assumed orientation we may say that the change in $\arg(f)$ along any level curve $L$ in $D$ is $2\pi{N}$ for some $N\in\{1,2,\ldots\}$ which is independent of $L$.

\begin{claim}
Let $\gamma$ be a closed path in $D$.  Then the change in $\arg(f)$ along $\gamma$ is an integer multiple of $2\pi{N}$.
\end{claim}

Let $k$ denote the number of times $\gamma$ winds around $E_1$.  Then since $D$ contains no zeros or poles of $f$, the Argument Principle implies that the change in $\arg(f)$ along $\gamma$ is $k$ times the change in $\arg(f)$ along $E_1$.  Thus the change in $\arg(f)$ along $\gamma$ is $k2\pi{N}$.

We now wish to define the conformal mapping described in the statement of the theorem.  Fix some $z_0$ in $D$ at which $f$ takes a positive real value.  (To see that such a point $z_0$ exists, observe that for any $z\in{D}$, the change in $\arg(f)$ along $\Lambda_z$ is $2\pi{N}$, so there are $N$ different points in $\Lambda_z$ at which $f$ takes positive real values.)  We wish to define a map $\phi:D\to\mathbb{C}$ which we will show is a conformal map with $\phi(D)=ann(0;{\epsilon_1}^{\frac{1}{N}},{\epsilon_2}^{\frac{1}{N}})$.  For $w\in{D}$, let $\gamma:[0,1]\to{D}$ be a path such that $\gamma(0)=z_0$ and $\gamma(1)=w$.  Define $\alpha$ to be the change in $\arg(f)$ along $\gamma$.  Then define $\phi(w)\colonequals|f(w)|^{\frac{1}{N}}e^{i\frac{\alpha}{N}}$.  This may be shown to be the value at $w$ of analytic continuation of $f^{\frac{1}{N}}$ along the path $\gamma$, and thus (if well defined) is analytic in a neighborhood of $w$.  Therefore if $\phi$ defined as such is well defined, then $\phi$ is analytic on $D$.  Let $w\in{D}$ be given, and let $\gamma_1$ and $\gamma_2$ be paths in $D$ from $z_0$ to $w$.  Let $\phi_1(w)$ denote the value of $\phi$ at $w$ obtained using $\gamma=\gamma_1$ and let $\phi_2(w)$ denote the value of $\phi$ at $w$ obtained by using $\gamma=\gamma_2$.  Let $\alpha_1$ denote the change in $\arg(f)$ along $\gamma_1$ and let $\alpha_2$ denote the change in $\arg(f)$ along $\gamma_2$.  Let $\gamma_3$ denote the path in $D$ obtained by traversing $\gamma_1$ and then traversing $-\gamma_2$.  $\gamma_3$ is a closed path in $D$, so by the claim above, the change in $\arg(f)$ along $\gamma_3$ (which is the change in $\arg(f)$ along $\gamma_1$ minus the change in $\arg(f)$ along $\gamma_2$) is an integer multiple of $2\pi{N}$.  Thus $\alpha_1=\alpha_2+k2\pi{N}$ for some $k\in\mathbb{Z}$.  Thus

\[
\phi_1(w)=|f(w)|^{\frac{1}{N}}e^{i\frac{\alpha_1}{N}}=|f(w)|^{\frac{1}{N}}e^{i\frac{\alpha_2+k2\pi{N}}{N}}= |f(w)|^{\frac{1}{N}}e^{i\frac{\alpha_2}{N}}e^{\frac{k2\pi{N}}{N}}=\phi_2(w).
\]

Therefore whether we define $\phi(w)$ using $\gamma_1$ or using $\gamma_2$ we obtain the same value.  (Note that we are essentially just showing that we may take an $N^{th}$ root of $f$ on $D$.)

\begin{claim}
$\phi(D)=ann(0;{\epsilon_1}^{\frac{1}{N}},{\epsilon_2}^{\frac{1}{N}})$.
\end{claim}

Note that by the Maximum Modulus Theorem, for each $z\in{D}$, since $D\subset{G}$ and $D$ does not contain any zero or pole of $f$, $|f(z)|\in(\epsilon_1,\epsilon_2)$, and thus $|\phi(z)|\in({\epsilon_1}^{\frac{1}{N}},{\epsilon_2}^{\frac{1}{N}})$.  Therefore $\phi(D)\subset{ann}(0;{\epsilon_1}^{\frac{1}{N}},{\epsilon_2}^{\frac{1}{N}})$.

Let $\xi\in{ann}(0;{\epsilon_1}^{\frac{1}{N}},{\epsilon_2}^{\frac{1}{N}})$ be given.  Choose some $x\in{D}$ such that $|f(x)|^{\frac{1}{N}}=|\xi|$.  Such an $x$ exists because $|f|\equiv\epsilon_1$ on $E_1$ and $|f|\equiv\epsilon_2$ on $E_2$, and $|f|$ is continuous.  Let $\rho\in[0,2\pi)$ be such that $\arg(\xi)=\arg(\phi(x))+\rho$.  (And thus $\phi(x)e^{i\rho}=\xi$.)  Since the change in $\arg(f)$ along $\Lambda_x$ is $2\pi{N}$ with $N\neq0$, there is some point $x'$ in $\Lambda_x$ such that if $\Lambda_x$ is traversed from $x$ to $x'$, then the change in $\arg(f)$ along this portion of $\Lambda_x$ is exactly $\rho{N}$.  One can easily then show that $\phi(x')=\xi$.

Therefore we conclude that $\phi(D)=ann(0;{\epsilon_1}^{\frac{1}{N}},{\epsilon_2}^{\frac{1}{N}})$.

\begin{claim}
$\phi$ is injective.
\end{claim}

Let $w_1,w_2\in{D}$ be given such that $\phi(w_1)=\phi(w_2)$.  From the definition of $\phi$, $|f(w_1)|=|f(w_2)|$.  Therefore $w_1$ and $w_2$ are in the same level curve of $f$.  The Maximum Modulus Theorem and Theorem~\ref{Two Level Curve Implies Crit. Level Curve.} together imply that there is only one level curve of $f$ in $D$ on which $|f|$ takes the value $|f(w_1)|$.  Thus we conclude that there is some level curve $L$ of $f$ in $D$ which contains both $w_1$ and $w_2$.  Let $\gamma_1$ be a path in $D$ from $z_0$ to $w_1$.  Let $\gamma_2$ be a path obtained by traversing $L$ in the direction of increase of $\arg(f)$ from $w_1$ to $w_2$.  (If $w_1=w_2$ let $\gamma_2$ be constant.)  Let $\gamma_3$ denote the path from $z_0$ to $w_2$ obtained by first traversing $\gamma_1$ and then traversing $\gamma_2$.  Let $\alpha_i$ denote the change in $\arg(f)$ along $\gamma_i$ for $i\in\{1,2\}$.  Since $\arg(f)$ is strictly increasing as $\gamma_2$ is traversed, and the total change in $\arg(f)$ as all of $L$ is traversed is $2\pi{N}$, we conclude that $\alpha_2\in[0,2\pi{N})$.  Since $\phi(w_1)=\phi(w_2)$, we have the following equation, calculating $\phi(w_1)$ using $\gamma_1$, and calculating $\phi(w_2)$ using $\gamma_3$.

\[
|f(w_1)|^{\frac{1}{N}}e^{i\frac{\alpha_1}{N}}=\phi(w_1)=\phi(w_2)=|f(w_2)|^{\frac{1}{N}}e^{i\frac{\alpha_1+\alpha_2}{N}}=|f(w_1)|^{\frac{1}{N}}e^{i\frac{\alpha_1}{N}}e^{i\frac{\alpha_2}{N}}.
\]

Dividing both sides by $|f(w_1)|^{\frac{1}{N}}e^{i\frac{\alpha_1}{N}}$, we obtain $1=e^{i\frac{\alpha_2}{N}}$.  Thus $\alpha_2$ is an integer multiple of $2\pi{N}$.  Since $\alpha_2\in[0,2\pi{N})$, we conclude that $\alpha_2=0$, and thus $w_1=w_2$.

It remains to show that $\phi$ extends continuously all points in the boundary of $D$ except possibly the zeros of $f'$ in $E_1$.  To see this we just observe that the definition of $\phi$ may be defined in the same way as above for all points in $\partial{D}$.  This extension is well defined by the same argument used above except at the zeros of $f'$ in $E_1$.  If $z$ is one of the zeros of $f'$ in $E_1$, then for any path $\gamma:[0,1]\to{G}$ such that $\gamma([0,1))\subset{D}$, and $\gamma(1)=z$, it may be shown using the continuity of $f$ that $\displaystyle\lim_{r\to1^-}\phi(\gamma(r))$ exists.
\end{proof}

\bibliographystyle{plain}
\bibliography{refs}

\end{document}